\documentclass[reqno,12pt,letterpaper]{amsart}

\usepackage{amsmath,amssymb,amsthm,graphicx,mathrsfs,url}
\usepackage[usenames,dvipsnames]{color}
\usepackage[colorlinks=true,linkcolor=Red,citecolor=Green]{hyperref}
\usepackage{amsxtra}

\relax

\numberwithin{equation}{section}

\def\arXiv#1{\href{http://arxiv.org/abs/#1}{arXiv:#1}}

\newcommand{\be}{\begin{equation}}
\newcommand{\ee}{\end{equation}}

\def\bbbone{{\mathchoice {1\mskip-4mu {\rm{l}}} {1\mskip-4mu {\rm{l}}}
{ 1\mskip-4.5mu {\rm{l}}} { 1\mskip-5mu {\rm{l}}}}}

\newcommand{\CI}{{\mathcal C}^\infty }
\newcommand{\CIc}{{\mathcal C}^\infty_{\rm{c}} }

\newcommand{\NN}{{\mathbb N}}

\newcommand{\RR}{{\mathbb R}}

\newcommand{\TT}{{\mathbb T}}

\newcommand{\ZZ}{{\mathbb Z}}

\newcommand{\supp}{\operatorname{supp}}

\theoremstyle{plain}

\newtheorem{thm}{Theorem}
\newtheorem{prop}{Proposition}[section]

\newtheorem{lem}[prop]{Lemma}

\theoremstyle{definition}

\numberwithin{equation}{section}

\def\bbbone{{\mathchoice {1\mskip-4mu {\rm{l}}} {1\mskip-4mu {\rm{l}}}
{ 1\mskip-4.5mu {\rm{l}}} { 1\mskip-5mu {\rm{l}}}}}

\def\squarebox#1{\hbox to #1{\hfill\vbox to #1{\vfill}}}

\usepackage{amsxtra}

\ifx\pdfoutput\undefined
  \DeclareGraphicsExtensions{.pstex, .eps}
\else
  \ifx\pdfoutput\relax
    \DeclareGraphicsExtensions{.pstex, .eps}
  \else
    \ifnum\pdfoutput>0
      \DeclareGraphicsExtensions{.pdf}
    \else
      \DeclareGraphicsExtensions{.pstex, .eps}
    \fi
  \fi
\fi

\title[Rough controls for Schr\"odinger operators on 2-tori]
{Rough controls for Schr\"odinger operators on 2-tori}

\author[N. Burq]{Nicolas Burq}
\address{Universit{\'e} Paris Sud, Universit\'e Paris-Saclay, CNRS
Math{\'e}matiques,
B{\^a}t 307, 91405
Orsay Cedex, France}
\email{Nicolas.burq@u-psud.fr}
\author[M. Zworski]{Maciej Zworski}
\address{Mathematics Department, University of California, Berkeley, \
CA 94720, USA}
\email{zworski@math.berkeley.edu}

\usepackage{amssymb}
\usepackage{amsmath, amsthm, amsopn, amsfonts}

\def\11{{\rm 1~\hspace{-1.4ex}l} }
\def\R{\mathbb R}

\def\Z{\mathbb Z}
\def\N{\mathbb N}

\def\T{\mathbb T}

\begin{document}

\begin{abstract}
The purpose of this note is to use the results and methods of 
\cite{BBZ} and \cite{BZ4} to obtain control and observability by rough functions and sets on 2-tori, $ \TT^2 = \RR^2/\ZZ \oplus \gamma \ZZ $. 
We show that for a non-trivial $ W \in L^\infty ( \TT^2 ) $, 
solutions to the  Schr\"odinger equation, $ ( i \partial_t + \Delta ) u = 0 $, satisfy
$ \| u |_{ t = 0 }
\|_{ L^2 ( \T^2 ) } \leq K_T \| W u \|_{ L^2 ( [0,T] \times \TT^2 )}  $.
%where the right hand side is finite thanks to Theorem \ref{th.BBZ}.
In particular, any Lebesgue measurable set of positive measure can be used for observability. This leads to controllability with 
localization functions in $ L^2 ( \TT^2 ) $ (or $ L^4$) and controls in 
$ L^4 ( [ 0 , T ] \times \TT^2 ) $ (or $ L^2 $).
 For continuous $ W $ this follows from the results of
Haraux \cite{Ha} and Jaffard \cite{Ja}, while 
for $ \TT^2 = \RR^2/ ( 2 \pi \ZZ )^2 $ and 
$ T > \pi $ this can be deduced from the results of Jakobson \cite{Jak}.

\end{abstract}   

\maketitle   

\section{Introduction}   
\label{in}
The purpose of this paper is to investigate the general question of control theory with localized control functions. When the localization is performed by a continuous function, the question is completely settled for wave equations~\cite{BLR,BG} and well understood for Schr\"odinger equations on tori~\cite{Ha, Ja, Ko, BZ4, AM}. 

In this paper we localize only to 
sets of {\em positive measure} or more generally use control functions in 
$ L^4 $. The understanding is then much poorer and only partial results are available even for the simpler case of wave equations~\cite{Bu16,Bu17}. Using the work with Bourgain \cite{BBZ} and \cite{BZ4} we completely settle  the question for Schr\"odinger equation on the two dimensional torus taking advantage, as in previous papers, of the
particular simplicity of the dynamical structure.

To state the control result consider
\begin{equation}
\label{eq:con}
\begin{gathered}  \T^2 : = \RR^2 /  \ZZ \times \gamma \ZZ \,, \ \ \gamma \in \RR \setminus \{
 0 \}\,,    \ \  a \in L^2 ( \T^2 ) , \\
 (i \partial_t + \Delta  ) u ( t, z) = a(z) 1_{(0,T)}f \,, \ \ \ 
u ( 0 , z ) =u_0 ( z ) , \end{gathered}
\end{equation}
where $a $ is a localisation function and $f$ a control. 
From \cite[Proposition 2.2]{BBZ} (see Theorem \ref{th.BBZ} below) we know that for $f \in L^{4}( \T^2 ; L^2(0,T))$ (so that  $af \in L^{4/3}( \T^2 ; L^2(0,T))$), and any $u_0 \in L^2( \T^2)$, there exists a unique solution 
$$ u \in C^0 ([0,T]; L^2( \T^2) ) \cap L^4 ( \T^2; L^2( 0,T)).$$ 

A classical question of control is to fix $ a $ and ask for which $u_0 \in L^2$ does there exist a control $f$ such that the solution of~\eqref{eq:con} satisfies $u |_{t>T} =0$? We show that on $\mathbb{T}^2$ it is always the case as soon as $a \in L^2$ is non-trivial:

\begin{thm}
\label{th.2}
Let $a \in L^2( \T^2), \| a \|_{L^2} >0$ and $ T > 0 $. 
Then for any $u_0\in L^2( \T^2)$ there exists $f \in L^4( \T^2 ; L^2(0,T))$ such that the solution $u$ of~\eqref{eq:con} satisfies $u |_{t=T} =0$.

If in addition $ a \in L^4 ( \T^2 ) $ then the same statement holds with 
$ f \in L^2 ( ( 0 , T ) \times \TT^2 ) $.
\end{thm} 

The next result shows that adding an $ L^2 $ damping term results in exponential decay:

\begin{thm}
\label{th.3}
 For $a\in L^2( \T^2)$, $ a \geq 0 $, $ \| a \|_{L^2} > 0$,  there exist $C, c>0$ such that for any $u_0 \in L^2( \T^2)$, the equation 
 \begin{equation}\label{damped} (i \partial_t + \Delta + ia) u=0, \qquad u |_{t=0} = u_0, 
\end{equation}
has a unique global solution $ u \in L^\infty( \R; L^2( \T^2))\cap L^4( \T^2; L^2_{\rm{loc}}(\R)) $ and 
\begin{equation}\label{damped}
 \| u\|_{L^2( \T^2)} (t) \leq C e^{-ct} \| u_0 \|_{L^2( \T^2)}.
 \end{equation}
\end{thm}

As shown in \S \ref{HUM} both results follow from an the observability estimate. We should think of $ a $ in Theorem \ref{th.2} as 
$ W^2 $ where $ W $ appears in the following statement:

\begin{thm}
\label{th.1}
Suppose that $ W \in L^4 ( \TT^2 ) $, $ \| W \|_{L^4} >0 $. Then for any $ T > 0$ there exists
$ K $ such that for $ u \in L^2 ( \TT^2 ) $, 
\begin{equation}
\label{eq:th1}   \| u \|_{ L^2 ( \TT^2 ) } \leq K \| W e^{ it \Delta} u \|_{ L^2 ( 
( 0 ,T )_t \times \TT^2 ) } .
\end{equation}
\end{thm}

To keep the paper easily accessible we present proofs in the case 
when $ \gamma \in \mathbb Q $ in \eqref{eq:con}.
Irrational tori require a more complicated reduction to rectangular 
coordinates -- see \cite[Lemma 2.7 and Fig.1]{BZ4} but the modification 
can be done as in that paper. The crucial 
\cite[Proposition 2.2]{BBZ} is valid for all tori. Another approach to 
treating (higher dimensional) irrational tori can be found in the work of 
Anantharaman--Fermanian-Kammerer--Maci\`a, see \cite[Corollary 1.19, Theorem 1.20]{AFM}. 

Since, as is already clear, \cite[Proposition 2.2]{BBZ} plays a crucial in many proofs we 
recall it in a version used here:

\begin{thm}
\label{th.BBZ}
Let $ T > 0 $. There exists $ C = C_T $ such that for 
\[  u_0\in L^2( \T^2), \ \ \ f \in  L^{\frac 4
  3 } (\T^2 ; L^2(0, T)),\]
the solution to 
$ (i \partial_t + \Delta) u =f$, $ u|_{t=0} = u_0$, 
 satisfies 
\[
 \| u\|_{L^\infty ( ( 0, T ) ; L^2 ( \TT^2 ) \cap L^4(\T^2; L^2( (0,T)))}\\
 \leq C \left(\|u_0\|_{L^2( \T^2)} + \|f\|_{ L^1 ( ( 0 ,T ) ; L^2 ( \TT^2 )) +  L^{\frac 4 3 } (\T^2 ; L^2(0, T))}\right).
\]
\end{thm}

\medskip

\noindent
{\bf Remarks.} 1. Theorem \ref{th.1} is equivalent to the same statement with $ W \in L^\infty ( \TT^2 ) $ (by replacing $ W \in L^4 $ by 
$ \bbbone_{ |W | \leq N } W \in L^\infty $ with $ N$ sufficiently large). Both the proof
and derivations of Theorems \ref{th.2} and \ref{th.3} are easier with the 
$ L^4 $ formulation.

\noindent 2. For rational tori and for $ T > \pi $, Theorem \ref{th.1},
 and by Proposition \ref{p:hum} below, Theorems \ref{th.2} and \ref{th.3},  follow from the results of Jakobson \cite{Jak}. That is done by using the 
complete description of microlocal defect measures for eigenfuctions of
$ \mathcal \RR^2 / 2 \pi \ZZ^2 $. We explain this in detail in the 
appendix.

\noindent
3. The starting point of \cite{Jak} and \cite{BBZ} was the classical inequality of Zygmund: 
\begin{equation}
\label{eq:zyg0} \forall \, \lambda \in \N, \ \ \| \sum_{{ |n|^2= \lambda}} c_ne^{in \cdot z} \|_{L^4( \T^2_z)}^2 \leq \frac{\sqrt 5}{ 2 \pi} 
\sum_{{ |n|^2= \lambda}} |c_n|^2 , \ \ \ z \in \TT^2 
= \RR^2 /  2 \pi \ZZ^2 , \ \  n \in \ZZ^2 . 
\end{equation}
In particular for $ \TT^2 = \RR^2/ 2 \pi \ZZ^2 $, we easily see how 
the homogeneous part ($f=0$) in Theorem \ref{th.BBZ}  follows from \eqref{eq:zyg0}. For that
put $ u = \sum_{\lambda } u_\lambda$, $ u_\lambda = \sum_{ 
_{ { |n|^2= \lambda}} } c_n e^{ i n \cdot x} $. 
 Then, using \eqref{eq:zyg0} in the third line,
\[ \begin{split} 
\| e^{ it \Delta } u \|_{ L^4 ( \TT^2, L^2 ( ( 0 , 2 \pi )) ) }^4 
& =  
\int_{ \TT^2 } \left( \int_{0}^{2\pi } \left| \sum_{ \lambda } e^{ it \lambda} 
u_\lambda ( z ) \right|^2 dt \right)^2 dz 
= (2 \pi)^2 \int_{ \TT^2 } \left( \sum_\lambda | u_\lambda ( z ) |^2 \right)^2 dz \\
& 
= ( 2 \pi)^2\int_{\TT^2 } \sum_{ \lambda, \mu} |u_\lambda ( z ) |^2 |u_\mu ( z ) |^2 dz \leq ( 2 \pi)^2\sum_{ \lambda, \mu } \| u_\lambda \|_{L^4}^2 \| u_\mu \|_{L^4}^2 
\\ & \leq 5  \sum_{ \lambda, \mu} \| u_\lambda \|_{L^2}^2 \| u_\mu \|_{L^2}^2 
= 5 \left( \sum_\lambda \| u_\lambda \|_{L^2}^2 \right)^2 = 
5  \| u \|_{L^2}^4. 
\end{split} \]
Generalizations for the time dependent Schr\"odinger equation 
in higher dimensions were obtained by 
A\"issiou--Jakobson--Maci\`a  \cite{AJM}.

\noindent
4. Other than tori, the only other   manifolds for which \eqref{eq:th1} is known for {\em any} non-trivial continuous $ W $ 
are compact hyperbolic surfaces. That was proved by 
Jin \cite{Ji} using results of Bourgain--Dyatlov \cite{BD} and Dyatlov--Jin
\cite{DJ}. 

\medskip\noindent\textbf{Acknowledgements.}
This research was partially supported by 
Agence Nationale de la Recherche through project ANA\'E ANR-13-BS01-0010-03
(NB), by the 
NSF grant DMS-1500852 and by a Simons Fellowship (MZ). MZ gratefully acknowledges the hospitality of Universit\'e Paris-Sud during the writing of this paper. We would also like to thank Alexis Drouot, Semyon Dyatlov and 
Fabricio Maci\`a for helpful comments on the first version.

\section{Semiclassical observability}

We follow the strategy of \cite{BZ4} and \cite{BBZ} and first prove a semiclassical
observability result. For that we define
\begin{equation}
\label{eq:Pih}  \Pi_{h, \rho} (u_0)  := \chi \left( \frac{ - h^2\Delta - 1} 
  {\rho} \right) u_0\,,  \ \ \ \rho > 0  \,,
 \end{equation}
 where $ \chi \in \CIc (( - 1 , 1 ) )$ is equal to $ 1 $ near $ 0 $.  With this notation the main result of this section is
\begin{prop}
\label{p:semi}
Suppose that $ a \in L^2 ( \TT^2 ) $, $ a \geq 0 $, $ \| a \|_{L^2} > 0 $.
For any $ T > 0 $ there exist $ K $, $ \rho_0 > 0 $ and $ h_0 > 0 $ such that
for any $ u_0 \in L^2 ( \TT^2 ) $, 
\begin{equation}
\label{eq:semi}
\| \Pi_{h, \rho} u_0 \|_{ L^2 }^2 \leq K \int_0^T \!\!\int_{ \TT^2} a ( z ) | e^{ i t \Delta} \Pi_{h, \rho} u_0  ( z ) |^2 dz dt,
 \end{equation}
 for $ 0 < \rho < \rho_0 $ and $ 0 < h < h_0 $.
\end{prop} 
The proof of the Proposition proceeds by contradiction: if \eqref{eq:semi} does not hold then there exists $ T > 0 $ such that for any $ n \in \NN $ there
exist $ 0 < h_n < 1/n $, $ 0 < \rho_n < 1/n $ and $ u_n \in L^2 $ for which \begin{equation}
\label{eq:contr}
1 = \| u_n  \|_{ L^2 }^2 > n \int_0^T \!\!\int_{ \TT^2} a ( z ) | e^{ i t \Delta}  u_n  ( z ) |^2 dz dt, \ \ u_n = \Pi_{ h_n , \rho_n } u_n .
\end{equation}
We will use semiclassical limit measures associated to subsequences of $ u_n$'s.

\subsection{Semiclassical limit measures} 
Each sequence $ u_n ( t ) := e^{ i t \Delta } u_n $, is bounded in $L^2_{\rm{loc}} ( \mathbb{R} \times \T^2)$. After possibly choosing a subsequence, $ u_n $'s define
a semiclassical defect measure $\mu $ on $\R_t \times T^* (\T^2_z)$ such
that for any function $\varphi \in 
{\mathcal C}^0_{\rm{c}}  ( \R_t)$ and any $A\in \CIc ( T^*\T^2_z)$, we have 
\begin{equation}
\label{eq:mu}
\begin{split} \langle \mu, \varphi(t) A(z, \zeta)\rangle &   = \lim_{n\rightarrow  \infty} \int_{\R_t \times \T^2} \varphi(t) 
\langle A(z, h_{n} D_z)  u_{n} ( t )   , u_n (t ) 
\rangle_{ L^2 ( \TT^2 ) } dt \,.
\end{split} 
\end{equation}
%Here we also defined the measure $ \mu_\varphi $ on $ T^* \TT^2 $, for
%any fixed $ \varphi \in {\mathcal C}_0^0 ( \RR ) $.

\renewcommand\thefootnote{\ddag}%

The measure $ \mu $ enjoys the following properties:
\begin{gather}\label{eq:propmu} \begin{gathered}
 \mu  (( t_0, t_1 ) \times T^*\T^2_z) = t_1 - t_0  , \ 
\ \
\supp \mu \subset \Sigma := \{(t,z, \zeta) \in \R_t \times \T^2_z \times \R^2_\zeta \; : \;  |\zeta|=1\}, \\
 %\zeta \cdot \nabla_x \mu =0, \ \text{ that is, } \ 
 \int_{\RR} \int_{ T^* \TT^2 } \varphi( t ) A ( z + s \zeta, \zeta ) d \mu = 0,
\ \ \ \varphi \in {\mathcal C}^0_{\rm{c}}  ( \R_t), \  A \in \CIc ( T^*\T^2_z), \end{gathered} \end{gather}
  see \cite{Ma} for the derivation and references.
  
We have an additional property which follows from an easy part of Theorem \ref{th.BBZ} (in the rational case related to the Zygmund inequality \eqref{eq:zyg0}): for any $ \tau \geq 0 $  there exists $ m_\tau 
\in L^2 ( \TT^2 ) $ such that for all $ f \in \mathcal C ( \TT^2 ) $
\begin{equation}
\label{eq:mT} 
\int_{0}^\tau \!\!\int_{ T^* \TT^2 } f ( z) d \mu ( t, z, \zeta ) = 
\int_{\TT^2 } m_\tau ( z ) f ( z) dz .
\end{equation}
In fact, Theorem \ref{th.BBZ} shows that 
\begin{equation}
\label{eq:UnT}
 U_n^\tau  ( z ) := \int_{0}^\tau | u_n ( t , z ) |^2  dt 
 \end{equation}
satisfies 
\begin{equation}
\label{eq:prop22} \| U^\tau_n \|_{ L^2 ( \TT^2_z ) } 
= \| u_n ( t, z ) \|_{ L^4 ( \TT^2_z , L^2 ( ( 0 , \tau ) ) }^2  \leq 
C \| u_n \|_{L^2 ( \TT^2 ) }^2 = C .
\end{equation}
But then, after passing to a subsequence, $ U_n^T $ converges weakly to 
$ m_\tau \in L^2 ( \TT^2 ) $. Since
\[ \int_{0}^\tau \int_{ T^* \TT^2 } f ( z) d \mu ( t, z, \zeta ) = \lim_{ n \to \infty } \int_{0}^\tau \int_{\TT^2} f ( z ) 
U_n^\tau ( z )  dz , \]
this proves \eqref{eq:mT}. 

The assumption \eqref{eq:contr} gives the following
\begin{lem}
\label{l:1}
Let $ m_\tau $ be defined by \eqref{eq:mT} with the measure $ \mu $ obtained from $ e^{i t\Delta } u_n $ satisfying \eqref{eq:contr}. Then
\begin{equation}
\label{eq:l1}
%\int_0^T \!\!\int_{T^*\TT^2 } a ( z ) d \mu ( t , z, \zeta ) =
 \int_{\TT^2 } a ( z ) m_T ( z ) dz = 0 .
\end{equation}
\end{lem}
\begin{proof}
We choose 
\begin{equation}
\label{eq:aj}  a_j \in \CI ( \TT^2 )  , \ \ a_j \geq 0 , \ \  
\lim_{ j \to \infty } \| a - a_j\|_{L^2 } = 0 .\end{equation}
 Using \eqref{eq:contr} and then \eqref{eq:prop22} (with the notation introduced in 
\eqref{eq:UnT}), 
\[  \begin{split} \int_{\TT^2} m_T( z )   a_j( z) dz & = 
\lim_{ n \to \infty }   \int_{\TT^2} U_n^T ( z ) ( a_j( z) - a ( z )  ) dz 
\\ & 
= \mathcal O ( \| U_n^T \|_{L^2 ( \TT^2) } ) \| a - a_j \|_{ L^2 ( \TT^2 ) } =
\mathcal O ( 1) \| a - a_j \|_{ L^2 ( \TT^2 ) } . 
\end{split} \]
Since $m_T \in L^2( \T^2)$, letting $ j \to \infty $ shows \eqref{eq:l1}.
\end{proof}

The next lemma shows that our measure has most of its mass on the set of rational directions:

\begin{lem}
\label{l:2}
Suppose that  $\mu $ is defined by $ u_n$ satisfying 
\eqref{eq:contr}. For $ m \in \NN $ define,
\[  W^m := \left\{ ( z , \zeta ) \in T^* \TT^2 \; : \; \zeta = \frac{ ( p , q ) }{
\sqrt{ p^2 + q^2 }} , \  \max( |p|, |q| )  \leq  m , (p, q) \in \Z^2, \ {\rm{gcd}}(p,q) = 1 \right\} , \]
its complement, $ W_m := \complement W^m $, 
and a measure $ \widetilde \mu_T $ on $ T^* \TT^2 $:
\begin{equation}
\label{eq:defwm}  % \langle \widetilde \mu_T , A \rangle := 
\int_{T^* \TT^2} A ( z, \zeta ) d\widetilde \mu_T  := \int_0^T \int_{T^* \TT^2 } 
A ( z, \zeta ) d\mu ( t, z , \zeta ), \ \  A \in \CIc ( T^* \TT^2 ) . 
\end{equation}
Then,
\begin{equation}
\label{eq:l2}
\forall \, \epsilon > 0 \ \exists \, m \ \text{ such that } \
\widetilde \mu_T ( W_m ) <  \epsilon . 
\end{equation}
\end{lem} 
\begin{proof}
 With $ a_j $'s from \eqref{eq:aj} we then have 
\begin{equation}
\label{eq:muTaj} \int_{T^* \TT^2 } a_j ( z ) d \widetilde \mu_T  ( z ,\zeta )   = \int_{ \TT^2 }
( a_j ( z ) - a ( z ) ) m_T ( z) dz = \mathcal O ( \| a - a_j \|_{L^2 } ) .
\end{equation}
With the notation
$  \langle b \rangle_S ( z, \zeta) :=  \frac{1}{S} \int_{ 0}^S 
b ( z + s \zeta ) ds $, $ b \in \CI  ( \TT^2 ) $, 
the last property in \eqref{eq:propmu} shows that for any $ S > 0 $,  
\[ \int_{T^* \TT^2 } a_j ( z ) d \widetilde \mu_T  ( z ,\zeta )   =
\int_{T^* \TT^2 } \langle a_j \rangle_S ( z , \zeta ) d \widetilde \mu_T  ( z ,\zeta ) . \]
We note that 
\begin{equation}
\label{eq:Wm}  W_{m+1} \subset W_m , \ \ 
W_\infty := \bigcap_{ m =1}^\infty W_m = \{( z , \zeta ) : |\zeta|=1, 
\zeta \in \RR^2 \setminus \mathbb Q^2 \} .
\end{equation}
For $ ( z, \zeta ) \in W_\infty $, unique ergodicity of the flow 
$ z \mapsto z + s \zeta $ shows that $ \langle a_j \rangle_S \to \langle a_j 
\rangle := \int_{\TT^2} a_j ( z ) dz / ( 2 \pi)^2. $ Fatou's Lemma then shows that 
\[ \begin{split}  \int_{W_\infty } a_j ( z ) d \widetilde \mu_T  ( z ,\zeta ) 
& =  \liminf_{ S \to \infty } \int_{W_\infty } \langle a_j \rangle_S ( z , \zeta ) d \widetilde \mu_T  ( z ,\zeta )  \\
& \geq \int_{W_\infty} \liminf_{ S \to \infty } \langle a_j \rangle_S ( z , \zeta ) d \widetilde \mu_T  ( z ,\zeta ) 
= \widetilde \mu_T ( W_\infty ) \langle a_j \rangle. \end{split} \]
Combining this with \eqref{eq:muTaj} shows that 
\[   \widetilde \mu_T ( W_\infty ) \leq \frac{ C \| a  - a_j \|_{L^2 } }
{\langle a_j \rangle} \rightarrow 0 , \ \ j \to \infty ,\]
(since $ \| a \|_{L^2 } > 0 $ and $ a \geq 0 $, $ \langle a_j \rangle 
\to \langle a \rangle >  0 $) which gives $ \widetilde \mu_T ( W_\infty ) = 0 $. But then \eqref{eq:Wm} implies that
$ \lim_{ m \to \infty } \widetilde \mu_T ( W_m ) = \widetilde \mu_T ( W_\infty ) = 0 $,
concluding the proof.
\end{proof}

\subsection{Reduction to one dimension}

We start with the following 
\begin{lem}
\label{l:RN}
Suppose that in \eqref{eq:l2} $ m $ is large enough so that
$ \widetilde \mu_T (  W_m ) < T  $ and that $ ( z, \zeta_0 ) \in \supp ( \widetilde \mu_T |_{ \complement W_m } ) $.
Then there exists $ F \in L^2 ( \TT^2 ) $ such that 
\begin{equation}  
\label{eq:muT0} 
 \widetilde \mu_T |_{ \TT^2 \times \{ \zeta_0 \} } = 
F  \otimes \delta_{ \zeta = \zeta_0 }  , \ \ \ \| F \|_{ L^2 ( \TT^2) } \neq 0 , \ \ F \geq 0 .
\end{equation}
\end{lem}
\begin{proof}
Let $ \pi : T^*\TT^2 \to \TT^2 $ be the natural projection map, $ \pi ( x, \xi) = x $. Then, using \eqref{eq:mT} and 
\eqref{eq:defwm},  for any Lebesgue measurable set $ E \subset \TT^2 $, 
\begin{equation}
\label{eq:FmT}  \pi_* ( \widetilde \mu_T |_{ \TT^2 \times \{ \zeta_0 \}}) ( E ) 
\leq  \pi_* ( \widetilde \mu_T ) ( E ) = \int_E m_T  ( z ) dz ,  \  \ \ m_T \in L^2. \end{equation}
The Radon--Nikodym theorem 
then shows that $ \pi_* ( \widetilde \mu_T |_{ \TT^2 \times \{ \zeta_0 \}})  = g m_T  $ where $ g $ is measurable, $m_T$-a.e. finite. The inequality \eqref{eq:FmT}
gives $ F := g m_T \leq m_T $ almost everywhere which shows that
$ F \in L^2 $. 
%We proceed as in the proof of \eqref{eq:mT} but with an additional localization. For that let $ \varphi \in \CIc ( \RR) $, $ \varphi ( 0 ) = 1 $. Then for $ \epsilon > 0 $ fixed,
%\begin{equation}
%\label{eq:gofz} \int_0^T  
%g ( z ) \left[ \varphi\left( \tfrac{ h_n D_z - \zeta_0 }\epsilon \right) u_n ( t, z )\right] \overline{ u_n ( t ,  z) } dz dt \longrightarrow 
%\int \varphi\left( \tfrac{ \zeta - \zeta_0 }\epsilon \right) g ( z ) 
%d \widetilde \mu_T ( z, \zeta ) . \end{equation}
%On the other hand, replacing $ U_n^\tau $ in \eqref{eq:UnT} by 
%\[ U_n^{T,\epsilon} ( z ) := \int_0^T \left[ \varphi\left( \tfrac{ h_n D_z - \zeta_0 }\epsilon \right) u_n ( t, z )\right] \overline{ u_n ( t ,  z) } dt \]
%and using Theorem \ref{th.BBZ} shows that,
%\[  \| U_n^{T, \epsilon } \|_{ L^2 ( \TT^z ) } \leq 
%\| \varphi \|_{L^\infty } \| u_n ( t , z ) \|_{L^4 ( \TT^2_z , L^2 ( 0 , T ) )} \leq C \| u_n\|_{L^2 ( \TT^2 ) } = C .
%\] 
%Hence a subsequence of $ U_n^{T, \epsilon } $ converges weakly to 
%$ F_\epsilon $ in $ L^2 $, $ \| F_\epsilon \|_{ L^2 } \leq C $, which combined with \eqref{eq:gofz} gives
%\[\int_{\TT^2 } F_\epsilon ( z ) g ( z ) dz  = \int \varphi\left( \tfrac{ \zeta - \zeta_0 }\epsilon \right) g ( z ) 
%d \widetilde \mu_T ( z, \zeta ) \to \int_{T^*\TT^2} g ( z ) d ( \widetilde \mu_T |_{ \zeta = \zeta_0 } )(z ,\zeta) , \ \ \epsilon \to 0 . \]
%Since for a sequence $ \epsilon_j \to 0 $, $ F_{\epsilon_j }  $ converge weakly to $ F \in L^2 $, 
%\eqref{eq:muT0} follows.
\end{proof}

Using \cite[Lemma 2.7]{BZ4} (see also \cite[Fig.1]{BZ4}) we can assume 
(by changing the torus but not $ \Delta_z $) that $ \zeta_0 = ( 0, 1 ) $, 
$ z = ( x, y ) $, $ x \in \RR/A_1 \ZZ $, $ y \in \RR /B_1 \ZZ $, $ A_1/B_1 \in 
\mathbb Q $. Abusing the notation we will keep the notation $ u_n $ and $ 
\mu $ for the transformed functions. The invariance property in 
\eqref{eq:propmu} and  the proof of Lemma \ref{l:1} show now that
\begin{equation}
\label{eq:muT00} 
\begin{gathered} 
 ( \widetilde \mu_T |_{ \TT^2 \times \{ ( 0,1 )  \} } ) = 
 g ( x) dxdy \otimes \delta_{0} ( \xi) \otimes\delta_1 ( \eta )  , \ \
 g \in L^2 ( \TT^1 ) , \ \ \| g \|_{ L^2 ( \TT^2 )} \neq 0 \\
\int_{ \TT^2 } g ( x ) a ( x, y ) d x dy = 0 .
\end{gathered}
 \end{equation}
Let us choose $ \chi \in \CIc ( \RR^2 )$ supported near $ | \zeta |= 1 $
and such that 
\begin{equation}
\label{eq:suppchi}
 \supp \chi \cap \complement W_m = \{ ( 0 , 1 ) \} . 
\end{equation}
We then define 
%\begin{equation}
%\label{eq:vn}  
$ v_n := \chi ( h D_z ) u_n $ and
$\nu :=  |\chi ( \zeta )|^2 \mu \neq 0 $. 
%\end{equation}
Definition \eqref{eq:mu} shows that $ \nu $ is the semiclassical defect measure associated to $ v_n ( t ) := e^{ i t \Delta } v_n = \chi ( h
D_z ) e^{ it \Delta }  u_n $ which in particular shows that 
\begin{equation}
\label{eq:beta} 
 \| v_n \|_{ L^2 ( \TT^2 ) }^2 = \nu ( [ 0 , T ] \times T^* \TT^2 ) \geq \beta :=  \int_{ \TT^2 } g ( x ) dx dy  > 0 .
\end{equation}

The reduction to a one dimensional problem is based, as in \cite{BZ4}, on 
a Fourier expansion in $ y $ (assuming $ B_1 = 2 \pi $ for notational simplicity):
\begin{equation}
\label{eq:Four}  v_n ( t ) ( x, y ) := [ e^{ i t \Delta } v_n]  (x, y) =
\sum_{ k \in \ZZ } [e^{ i t \partial_x^2 } v_{n , k }] ( x ) e^{-it k^2 + i ky } \end{equation}

We will now use a one dimensional result proved in \S \ref{1D} below:
\begin{lem}
\label{l:3}
Suppose that $ b \in L^1 ( \TT^1 ) $,  $ b\geq 0 $, $\| b \|_{ L^1 } > 0 $
and that $ T > 0 $. Then there exists 
$ C $ such that for $ w \in L^2 ( \TT^1 ) $,
\begin{equation}
\label{eq:l3}
\| w \|_{L^2 ( \TT^1 ) }^2 \leq C \int_0^T  b ( x) [e^{ i t \partial_x^2} w](x )|^2 dx dt .
\end{equation}
\end{lem}

Let $ a_j $ be again given by \eqref{eq:aj}. 
We apply \eqref{eq:l3} to \eqref{eq:Four} with 
$ b = \langle a  \rangle_y := \frac{1}{ 2 \pi} \int_{ \TT^1} a ( x, y ) dy $.
That gives,
\begin{equation*}
\begin{split} 
0 < \beta & = \| v_n \|_{L^2 ( \TT^2 ) }^2 = 
2 \pi \sum_{k \in \ZZ } \| v_{n,k} \|_{ L^2 ( \TT^1) }^2 
\leq C'  \int_0^T \langle a \rangle_y ( x ) \sum_{ k \in \ZZ} | [ e^{ it \partial_x^2 } 
v_{n,k} ] ( x ) |^2 dx dt 
\\
& = C \int_0^T \int_{ \TT^2 } \langle a  \rangle_y ( x ) 
[ e^{ it \Delta } v_n ] ( x, y ) dx dy  dt \\
& = C \int_0^T \int_{\TT^2} \langle a_j \rangle_y ( x) [ e^{ it \Delta } v_n ] ( x, y ) dx dy  dt + 
\mathcal O ( \| a - a_j \|_{L^2 (
\TT^2 ) }) \\
& \longrightarrow C  \int_{T^* \TT^2 }
\langle a_j \rangle_y ( x ) d \, \widetilde \nu_T ( x, y , \xi, \eta ) + \mathcal O ( \| a - a_j \|_{L^2 (
\TT^2 ) })   , \ \ 
n \longrightarrow  \infty ,
\end{split}
\end{equation*}
where $ \widetilde \nu_T = |\chi ( \xi, \eta )|^2 \widetilde \mu_T $ 
(see \eqref{eq:defwm}). In particular for every $ j $,  
\begin{equation}
\label{eq:ajn}
0 < \alpha \leq \int_{T^* \TT^2 }
\langle a_j \rangle_y ( x ) d \, \widetilde \nu_T ( x, y , \xi, \eta )  +  \mathcal O ( \| a - a_j \|_{L^2 (
\TT^2 ) } )   , \ \
\alpha := \frac{ \beta}{ C }.
\end{equation}
We now decompose the integral in \eqref{eq:ajn} as $
I_1 + I_2 $ and use \eqref{eq:muT00}:
\begin{equation}
\label{eq:aj1} 
\begin{split}  
I_1 := \int_{ \TT^2 \times \{ ( 0 , 1 ) \} } \langle a_j \rangle_y ( x ) d \, \widetilde \nu_T ( x, y , \xi, \eta ) & = 
\int_{ \TT^2 }  g ( x ) a_j ( x, y ) dx dy \\
& = 
\int_{ \TT^2 } g ( x ) ( a_j ( x, y ) - a ( x, y ) ) dx dy \\
& \leq {\sqrt{ 2 \pi }} \| g \|_{ L^2 ( \TT^1)  }  \| a_j - a \|_{L^2 ( \TT^2)  }  . 
\end{split}
\end{equation}
We now use use \eqref{eq:l2} and \eqref{eq:suppchi} to estimate the remainder:
\[ \begin{split} 
I_2 & := \int_{ ( \xi , \eta ) \neq ( 0 , 1 ) } 
\langle a_j \rangle_y ( x ) d \, \widetilde \nu_T ( x, y , \xi, \eta )
\leq \int_{W_m } \langle a_j \rangle_y ( x ) d \, \widetilde \nu_T ( x, y , \xi, \eta ) \\
& \leq \| a_j \|_{ L^\infty } \, \widetilde \nu_T ( W_m ) 
 \leq \epsilon  \| a_j \|_{ L^\infty },  .
\end{split}
\]
We now combine these two estimates with \eqref{eq:ajn} to obtain:
\[  0 < \alpha \leq K \| a_j - a \|_{L^2  ( \TT^2 ) } + \epsilon \| a_j \|_{ L^\infty ( \TT^2 ) },  \]
where the constant $ K $ depends on $ a $, $ u_n $ and $ \zeta_0 = ( 0, 1 ) $ but not on $ \chi $ and $ m $. 
Hence, we first choose 
$ j $ large enough so that $ K \| a_j - a \|_{L^2  ( \TT^2 ) } < \alpha/2  $ and then $ m $ large enough
and $ \chi $ satisfying \eqref{eq:suppchi} so that 
$ \epsilon \| a_j \|_{ L^\infty } < \alpha/2 $. This provides a contradiction 
and proves Proposition \ref{p:semi}.

\renewcommand\thefootnote{\dag}%

\subsection{One dimensional estimate}
\label{1D}

We now prove Lemma \ref{l:3}. The semiclassical part proceeds along the lines of the proof of Proposition \ref{p:semi}. The derivation of 
\eqref{eq:l3} from the semiclassical estimate follows the same arguments needed in \S \ref{obs} and we will refer to that section for details.

\begin{proof}[Proof of Lemma \ref{l:3}]
We start with a semiclassical statement: for every $ T $ there exist $ K$, 
$ \rho_0 $ and $ h_0 $ such that for $ 0 < h < h_0 $ and $ 0 < \rho < \rho_0 $
we have the analogue of \eqref{eq:semi}:
\begin{equation}
\label{eq:semi1}
\begin{gathered}
\| \pi_{h, \rho} u_0 \|_{ L^2 ( \TT^1) }^2 \leq K \int_0^T \!\!\int_{ \TT^1}b ( z ) | e^{ i t \Delta} \pi_{h, \rho} u_0  ( z ) |^2 dz dt,
\ \
\pi_{h, \rho} (u_0)  := \chi \left( \frac{  h^2 D_x^2  - 1} 
  {\rho} \right) u_0 \,.
  \end{gathered}
 \end{equation}
We proceed by contradiction which leads to an analogue of \eqref{eq:contr}
and then to a measure $ \omega_T  $ analogous to $ \widetilde \mu_T $ 
(see \eqref{eq:defwm}) on $ T^* \TT^1 $ and satisfying:
$  \supp \omega_T \subset \{ \xi = \pm 1 \} $, 
 $\partial_x \omega_T = 0$ ,
where the derivative is taken in the distributional sense.  
 
From \cite[Proposition 2.1]{BBZ}\footnote{See 
\url{https://math.berkeley.edu/~zworski/corr_bbz.pdf} for a corrected 
version.} and the argument in Lemma \ref{l:RN} (with weak convergence in 
$L^2 $ replaced by the weak$*$ convergence  in $ L^\infty = ( L^1)^* $) 
we obtain
\[  d \omega_T = \sum_{ \pm } f_\pm (x) dx \otimes \delta_{\pm 1} ( \xi) d\xi, \ \ \  f_\pm \in L^\infty (\TT^1) , \ \ f_\pm  \geq 0 . \]
But the fact that $ \partial_x \omega_T = 0 $ and the analogue of
Lemma \ref{l:1} show that 
$ f_\pm ( x ) = c_\pm \geq 0 $, $c_+ + c_- > 0 $ ,  
 $( c_+ + c_- )\int _{\TT^1 } b ( x ) dx = 0$, 
which is a contradiction proving \eqref{eq:semi1}.

From the semiclassical estimate we obtain 
\[ \| u_0\|_{L^2 ( \TT^1 ) } \leq C 
 \int_0^T \!\!\int_{ \TT^1}b ( z ) | e^{ i t \Delta}  u_0  ( z ) |^2 dz dt
 +  C \| u_0 \|_{ H^{-2}( \TT^1 )} .\]
 That is done by the same argument recalled in \S \ref{dya} below. 
 Finally the error term $ \| u_0 \|_{ H^{-1}( \TT^1 ) } $ is removed -- see 
\S \ref{eli} for review of the procedure for doing (applying \cite[Proposition 2.1]{BBZ} again).  
 \end{proof}

\section{Observability estimate}
\label{obs}
To prove Theorem \ref{th.1} we first prove a weaker statement involving an error term:

\begin{prop}
\label{p:weak}
Suppose that $ W \in L^4 ( \TT^2 ) $, $ a \geq 0 $ and 
$ \| W \|_{ L^4 } \neq 0 $.  Then for any $ T > 0 $, there exists
$ K $ such that for $ u \in L^2 $,
\begin{equation}
\label{eq:weak}
\| u \|_{ L^2 ( \TT ) }^2 \leq C \int_0^T \int_{\TT^2 } 
| W(z) e^{ i t \Delta } u ( z ) |^2 dz dt + C \| u \|_{ H^{-2} ( \TT^2 ) }.
\end{equation}
\end{prop}

\subsection{Dyadic decomposition}
\label{dya}
The proof of \eqref{eq:weak} uses a dyadic decomposition as in 
\cite[\S 5.1]{BBZ} and \cite[\S 4]{BZ4} and we recall the argument adapted to the setting of this paper. For that let
$ 1= \varphi_{0} (r)^2 + \sum_{k=1}^{\infty} \varphi_{k} ( r )^2$, where
\begin{gather*} % 1= \varphi_{0} (r)^2 + \sum_{k=1}^{\infty}
  \varphi_{k} ( r ) :=   \varphi(R ^{-k} |r| ), \enspace R > 1, \quad \varphi \in \CIc ( ( R^{-1}, R ) ; [0, 1]  ), \quad  (
R^{-1}, R) \subset \{ r \; : \;  \chi( r / \rho)  \geq
\textstyle{\frac12} \},
\end{gather*}
with $ \chi $ and $ \rho $ same as in \eqref{eq:Pih} and \eqref{eq:semi}.
Then,  we decompose $u_0$
dyadically:
$
\|u_0 \|^2_{L^2} = \sum_{k=0}^\infty  \|\varphi_{k}( - \Delta ) u_0
\|_{L^2}^2, 
$
which will allow an application of Proposition \ref{p:semi}.

\begin{proof}[Proof of Proposition \ref{p:weak}]
Let $\psi \in \CIc ( ( 0, T ) ; [ 0, 1 ] ) $ satisfy  $\psi ( t ) > 1/2$,
on $ T/3 < t < 2T / 3 $.
Proposition~\ref{p:semi} applied with $ a = W^2 $ shows that 
\begin{equation}\label{eq:piece}
\|\Pi_{h, \rho } u_0\|_{L^2}^2 \leq K \int_\RR \psi( t )^2 \| W \mathrm{e}^{\mathrm{i} t  \Delta } \Pi_{h, \rho } u_0\|_{L^2 ( \TT^2 ) }^2 dt, \quad 0 < h < h_0.
\end{equation}
Taking $K $ large enough so that $R^{-K } \leq h_{0}$ we apply
\eqref{eq:piece} to the dyadic pieces:
\begin{align*}
\|u_0 \|^2_{L^2} & =\sum_{ k \in \ZZ }  \|\varphi_{k}( -\Delta)u_0\|_{L^2}^2 \\
&\leq \sum_{k=0}^{K }  \|\varphi_{k}( -\Delta )u_0\|_{L^2}^2 +C
\sum_{k=K +1}^\infty  \int_{0}^T \psi ( t ) ^2 \| W \varphi_{k}(  -\Delta)\,
\mathrm{e}^{\mathrm{i}t \Delta } u_0\|_{L^2( \TT^2) } ^2 dt \\
& = \sum_{k=0}^{K }  \|\varphi_{k}( -\Delta )u_0\|_{L^2}^2 + C
\sum_{k=K +1}^\infty  \int_\RR \| \psi ( t ) W \varphi_{k}(  D_t )
\,\mathrm{e}^{\mathrm{i}t \Delta } u_0\|_{L^2( \TT^2) } ^2 dt.
\end{align*}
In the last equality we used the equation and 
replaced $ \varphi ( -\Delta ) $ by $
\varphi ( D_t ) $. 

We need to consider the
commutator of $  \psi \in \CIc ( ( 0,  T) ) $ and $
\varphi_{k} ( D_t )\,{=}\,\varphi ( R^{-k}  D_t ) $. If  $
\widetilde \psi \in \CIc ((0,T))$  is equal to $ 1 $ on $
\supp  \psi$ then the semiclassical pseudo-differential
calculus with $ h = R ^{-k}$ (see for instance
\cite[Chapter~4]{EZB}) gives
\begin{equation}\label{eq:wideps}
\psi ( t )  \varphi_{k} ( D_t )  =    \psi ( t )  \varphi_{k}
(D_t )   \widetilde  \psi ( t ) +  E _k ( t, D_t),\quad
\partial^\alpha E_k   = {\mathcal O}   ( \langle t \rangle^{-N } \langle
\tau \rangle^{-N}
R ^{- N k } ),
\end{equation}
for all $ N$ and uniformly in $ k $.

The errors obtained from $ E_k $ can be absorbed into the $
\| u_0 \|_{ H^{-2}  ( \TT^2 ) } $ term on the right-hand
side. Hence we obtain,
\begin{align*}
\|u_0 \|^2_{L^2} &\leq C \|u_0 \|_{H^{-2} ( \T^2)} ^2+ C \sum_{k=0}^\infty \int_{0}^T
 \|   \psi (  t)  \varphi_{k}(  D_t  )\, W \mathrm{e}^{ \mathrm{i}t \Delta }u_0 \|^2_{L^2( \TT^2)} dt\\
& \leq\widetilde C \|u_0 \|_{H^{-2} ( \T^2)} ^2+ K   \sum_{k=0}^\infty
\langle   \varphi_{k}(  D_t  )^2  \widetilde \psi ( t ) \, W \mathrm{e}^{\mathrm{i}t \Delta }u_0,  \widetilde \psi( t )\,
W \mathrm{e}^{\mathrm{i}t \Delta}u_0,  \rangle_{L^2 ( \RR_t \times \TT^2)} \\
& \leq  \widetilde  C \|u_0 \|_{H^{-2} ( \T^2)} ^2 + K \int_\RR  \| \widetilde \psi ( t ) \, W \mathrm{e}^{\mathrm{i}t \Delta} u_0\|^2_{L^2( \TT^2)} dt  \\
& \leq \widetilde C \|u_0 \|_{H^{-2} ( \T^2)} ^2 + K \int_{0}^T  \|\, W\mathrm{e}^{\mathrm{i}t \Delta } u_0\, \|^2_{L^2( \TT^2)} dt,
\end{align*}
where the last inequality is \eqref{eq:weak} in the statement of the
proposition.
\end{proof}

\subsection{Elimination of the error term}
\label{eli}

We now eliminate the error term on the right hand side of \eqref{eq:weak}.
For that we adapt the now standard method of Bardos--Lebeau--Rauch \cite{BLR} just we did at the end of \cite[\S 4]{BZ4}. The argument recalled there shows that if
\begin{equation}
\label{eq:defN0} N := \{ u \in L^2 ( \TT^2 ) \; : \; W e^{ i t \Delta } u \equiv 0 
\ \text{ on $ ( 0 , T ) \times \TT^2 $}\} 
\end{equation}
is non-trivial then since $ iW e^{it\Delta } \Delta u= \partial_t W e^{ i t \Delta } u \equiv 0$ on $ ( 0 , T ) \times \TT^2 $, then $N$ is invariant by the action of $\Delta$, and hence  
 it contains a nontrivial $ w \in L^2 (\TT^2 ) $ such that for some $ \lambda $,
\[ ( - \Delta - \lambda) w = 0 , \ \ \ W w \equiv 0 . \]
But then $ w $ is a trigonometric polynomial vanishing on a set of 
positive measure which implies that $ w \equiv 0 $. Hence
\begin{equation}
\label{eq:N0}
N = \{ 0 \} .
\end{equation}

\begin{proof}[Proof of Theorem \ref{th.1}]
Suppose the conclusion \eqref{eq:th1} were not to valid. Then 
there exists a  sequence $ u_n  \in L^2 ( \TT^2 ) $ such that
\begin{equation}
\label{eq:unW}  \| u_n \|_{ L^2 ( \TT^2 ) } = 1 , \ \ \ 
\| W e^{it \Delta } u_n \|_{ L^2 ( ( 0 , T ) \times \TT^2 ) }\to 0, \ \ n 
\to \infty . \end{equation}
By passing to a subsequence we can then assume that 
$ u_n $  converging weakly in $ L^2  ( \TT^2 ) $ and strongly in $ H^{-2} 
( \TT^2 ) $ to some $ u \in L^2 $.  From Proposition \ref{p:weak} we would also have
\[ 1 = \| u_n \|_{ L^2 ( \TT^2) }^2 \leq C 
\int_0^T \| W e^{it \Delta } u_n \|_{ L^2 ( \TT^2 ) }^2 dt + C \| u_n \|^2_
{H^{-2} ( \TT^2 ) } . \]
Hence, 
\begin{equation}
\label{eq:cont1} 1 \leq C \lim_{n \to \infty } \| u_n \|_{H^{-2}  ( \TT^2 ) } = 
C\| u \|_{ H^{-2} ( \TT^2 ) } 
 \ \ \Longrightarrow \ \ u \not \equiv 0 . \end{equation}
Let $ W_j \in \CI ( \TT^2 ) $ satisfy $ \| W - W_j \|_{L^4 ( \TT^2 ) }
\to 0 $. For $ \varphi \in \CIc ( ( 0, T ) \times \TT^2 ) $, due to 
distributional convergence, Theorem \ref{th.BBZ} and
\eqref{eq:unW},  
\[ \begin{split} |\langle W_j e^{ it\Delta } u , \varphi \rangle| & = \lim_{ n \to \infty }
|\langle e^{ it \Delta } u_n , W_j  \varphi \rangle|  \leq 
\lim_{ n \to \infty }
 \left( | \langle ( W_j - W ) e^{ i t \Delta } u_n  , \varphi \rangle | + 
\langle W e^{ it \Delta } u_n , \varphi \rangle \right) \\ & \leq
\|\varphi \|_{ L^2  }   \| ( W_j - W ) e^{it \Delta } u_n \|_{L^2 ( ( 0 , T ) \times 
\TT^2 ) }
 \leq C\|\varphi \|_{L^2 }  \|  W_j - W  \|_{L^4( \TT^2 ) } . 
\end{split}
\]
On the other hand the same argument shows that 
\[ |  \langle W e^{it \Delta  }u, \varphi \rangle | \leq 
| \langle  W_j  e^{ it \Delta } u ,   \varphi \rangle | + 
C \| \varphi\|_{L^2}  \|  W_j - W  \|_{L^4( \TT^2 ) } . \]
Combining the two inequalities we see that
$ |  \langle W e^{it \Delta  }u, \varphi \rangle | \leq C
\lim_{ j \to \infty } \|  W_j - W  \|_{L^4( \TT^2 ) } = 0 . 
$ which means that $ W e^{ it \Delta } u \equiv 0 $. Thus $ u \in N $ given by \eqref{eq:defN0}
and by \eqref{eq:N0}, $ u = 0 $. This contradicts \eqref{eq:cont1} 
completing the proof.

\end{proof}

\section{The HUM method: proofs of Theorems \ref{th.2} and \ref{th.3}}
\label{HUM}
We now show the equivalence of the stabilization, control and observability properties in our context. The proof is a variation on the classical HUM method \cite{Li}, but since our damping and localization functions are {\em not} in $L^\infty$ it requires additional care. 
\begin{prop}
\label{p:hum}
The following are equivalent (for fixed $T>0$).
\begin{enumerate}
\item \label{i} Let $a \in L^2( \T^2;\RR), \| a \|_{L^2} >0$. 
Then for any $u_0\in L^2( \T^2)$ there exists $f \in L^4( \T^2 ; L^2(0,T))$ such that the solution $u$ of~\eqref{eq:con} satisfies $u|_{t=T} =0$
\item \label{ii} Let $a \in L^4( \T^2;\RR), \| a \|_{L^4} >0$. 
Then for any $u_0\in L^2( \T^2)$ there exists $f \in L^2( (0,T)\times \T^2)$ such that the solution $u$ of~\eqref{eq:con} satisfies $u |_{t=T} =0$
\item \label{iii}Let $a \in  L^4( \T^2;\RR), \| a \|_{L^4} >0$. Then there exists $C>0$ such that for any $v_0 \in L^2( \T^2)$,
\begin{equation}
\label{observation}
\| v_0 \|_{L^2( \T^2)}^2 \leq C \|  a e^{it \Delta} v_0 \| _{L^2( (0,T) \times \T^2)}.
\end{equation}
\end{enumerate}
 
\end{prop}
\begin{proof}
Let us prove that ~(\ref{i}) implies~(\ref{ii}). Indeed, for $a\in L^4$, we can apply (\ref{i}) to $a^2\in L^2 $ and get a function $g \in L^4( \T^2 ; L^2(0,T))$ such that $a^2 g$ drives the system to rest, and (\ref{ii}) follows by defining $f= ag \in L^2( \T^2; L^2( 0,T)).$

To prove that (\ref{ii}) and (\ref{iii}) are equivalent, we follow the HUM method. Define the map 
$$ R: f \in L^2((0,T) \times \T^2) \mapsto R f = u |_{t=0},$$
where $u$ is the solution of the final value problem
$$  (i \partial_t + \Delta  ) u ( z) = a(x) 1_{(0,T)}f \in L^{4/3}( \T^2; L^2(0,T)), \quad u |_{T=0} =0.$$
By Theorem \ref{th.BBZ} $ R :  L^{4/3}( \T^2; L^2(0,T)) \to 
L^2 ( \TT^2 ) $ and  
\begin{equation}
\label{eq:2} 
{\text{(\ref{ii})}} \ \ \Longleftrightarrow \ \ 
R ( L^{4/3}( \T^2; L^2(0,T)) ) = L^2 ( \T^2 ) .
\end{equation}

Again by Theorem \ref{th.BBZ}, $e^{it \Delta} v_0 \in 
L^4( \T^2; L^2(0,T)) $ for $ v_0 \in L^2 ( \TT^2 ) $, we define 
$$S: v_0 \in L^2( \T^2) \mapsto 1_{(0,T)} \times a e^{it \Delta} v_0 \in L^2((0,T) \times \T^2),$$
and
\begin{equation}
\label{eq:3} 
{\text{(\ref{iii})}} \ \ \Longleftrightarrow \ \ 
\exists \, K \ \forall\, v_0 \in L^2 ( \TT^2) \ \
\| v_0 \|_{L^2 ( \TT^2 ) } \leq K\| S v_0 \|_{L^2 ( ( 0 , T ) \times \T^2 )}.\end{equation}
To relate $ R $ and $ S $ we integrate by parts: 
\[ \begin{split}
\int_0^T \int_{\T^2} af \overline {v} dx dt  &= \int_0^T \int_{\T^2} (i \partial_t + \Delta) u \overline {v} dx dt
%\\ & 
=  i \left[ \int_{\T^2} u \overline{v} dx \right] _0 ^T +\int_0^T \int_{\T^2}  u \overline {(i \partial_t + \Delta)v} dx dt
\\ & = -i \int_{\T^2} u \overline{v} dx|_{t=0} ,
\end{split}\]
which is the same as
\begin{equation}
\label{eq:SR} \bigl( f, Sv_0\bigr)_{L^2((0,T)\times \T^2)} = -i \bigl( Rf, v_0\bigr)_{L^2( \T^2)}.
\end{equation}

 Let us assume (\ref{ii}).  By \eqref{eq:2} and the closed graph theorem there exists $\eta>0$ such that the image of the unit ball in $L^2((0,T) \times \T^2)$ by $R$ contains the ball $ \{ v_0 \in L^2 ( \T^2 ) : 
 \| v_0 \|_{L^2} \leq \eta \} $. Hence for all $ v_0 \in L^2 ( \T^2 ) $
 there exists $ f \in L^2((0,T) \times \T^2)$, such that 
 \[   \| f \|_{ L^2 ((0,T) \times \T^2) } \leq \frac 1 \eta \| v_0\|_{ L^2}
 , \ \ \ R f = v_0 .\]
Hence, using \eqref{eq:SR}, 
\begin{equation}
\begin{split}
 \| v_0 \| _{L^2( \T^2)}^2 & = i  \bigl( f, Sv_0\bigr)_{L^2((0,T)\times \T^2)} \leq \| f\|_{L^2((0,T)\times \T^2} \| Sv_0 \|_{L^2((0,T)\times \T^2)}
  \\
 & \leq \frac 1 \eta \| Sv_0 \|_{L^2((0,T)\times \T^2)}
  \| v_0 \| _{L^2( \T^2)} ,
\end{split}
\end{equation}
and by \eqref{eq:3},  {(\ref{iii})} follows.

On the other, assume that {(\ref{iii})} holds. By \eqref{eq:3}, 
the operator 
$$-i R \circ S : L^2( \T^2) \rightarrow L^2( \T^2)$$ is continuous and, 
by \eqref{eq:SR}, 
there exists $ C > 0 $ such that for all $ v_0 \in L^2( \T^2) $, 
$$ \bigl( -i R \circ S v_0, v_0 \bigr) _{L^2( \T^2)} = \bigl( Sv_0, S v_0\bigr)_{L^2( (0,1) \times \T^2) } \geq \frac 1 C \| v_0 \|_{L^2( \T^2)}^2.$$
Consequently $ -i R \circ S  $ is an injective bounded self-adjoint operator,
hence bijective. 
This in turn shows that $R$ is surjective and
in view of \eqref{eq:2}  proves~(\ref{ii}). 

We also deduce that in~(\ref{ii}) we can assume that $f$ is of the form
$ f =  Su_0 = a e^{it \Delta} u_0$, which, changing $a$ to $a^2$ and using that   $e^{it \Delta} u_0 \in L^4(\T^2; L^2(0,1))$ implies ~(\ref{i}) when $a\geq 0$. By changing $f$ by a phase factor gives the general case of
(\ref{i}).
\end{proof}

In view of Theorem \ref{th.1} this proves Theorem \ref{th.2} and provides 
some additional versions of it.
We now turn to the damped Schr\"odinger equation. 

\begin{proof}[Proof of Theorem \ref{th.3}]
For $a  \in L^2$ with $ a \geq 0 $ and 
$H := (- i \Delta +a)$ we have
$$ \bigl( Hu, u\bigr)_{L^2} = \int_{\T^2} a |u|^2 (x) dx \geq 0 , \ \ u \in H^2 ( \T^2) . $$
Hence for $\lambda >0$ the equation 
$(H +  \lambda ) u =f \in L^2(\T^2)$ can be solved with 
$ \| f \|_{ L^2 } \leq \lambda^{-1} \| u \|_{L^2 } $. 
Hille--Yosida theorem then shows that $ H $ defines a strongly 
continuous semigroup
$ [0, \infty ) \ni t \mapsto \exp ( - t H ) $. Furthermore, when $u_0 \in H^2$, 
\[  u( t ) := \exp ( - t H ) u_0 \in C^1( [0 , \infty ) ; L^2( \T^2))\cap C^0( [ 0 , \infty ) ; H^2( \T^2)) .\]
We then check that 
\begin{equation}
\label{damping}
\begin{gathered}
\| u(t) \|^2_{L^2( \T^2)} = \| u_0 \|^2_{L^2( \T^2)}- \int_0 ^t \int_{\T^2} a(x)  |u|^2(s,x) dx ds,\\
u(t) = e^{it \Delta} u_0 + \int_0^t e^{i(t-s)\Delta} (au)(s) ds.
\end{gathered}
\end{equation}
Let $a_j \in C^0 ( \TT^2) $ and 
$  \| a_j -  a\|_{L^2 ( \TT^2)  } \to 0 $, $ j \to \infty $. Using 
the second expression in \eqref{damping},
\[ 
\begin{split}
\| u\|_{L^4(\T^2; L^2(0,T))} & \leq C \| u_0 \|_{L^2( \T^2)} + \| a_j u \|_{L^1((0,T); L^2( \T^2))} + \| (a- a_j)u \|_{L^{4/3}(\T^2; L^2( 0,T))}\\
& \leq C \| u_0 \|_{L^2( \T^2)} +C T\| a_j\|_{L^\infty} \| u\|_{L^{\infty}((0,T); L^2( \T^2))} 
\\ & \ \ \ \ \ \ \ \ \ \ \ \ \
 +C\| a-a_j\|_{L^2( \T^2)} \| u \|_{L^{4}(\T^2; L^2( 0,T))}
\end{split}
\]
Taking $ j$ large enough so that 
$C\| a-a_j\|_{L^2( \T^2)} \leq \frac12$, we get  
$$ \| u\|_{L^4(\T^2; L^2(0,T))} \leq C' \| u_0 \|_{L^2( \T^2)} , \ \
u_0 \in H^2 ( \T^2 ) . $$
Since $ H^2  $ is dense in $L ^2 $, this remains true for initial data
$u_0\in L^2( \T^2)$ and consequently, for $ a \in L^2 $, we get that 
\begin{equation}\label{esti-ap}
 \int_0 ^t \int_{\T^2} a(x)  |u|^2(s,x) dx ds\leq C \| u_0 \|_{L^2( \T^2)}.
 \end{equation}
By simple integration by parts~\eqref{damping} is true  for $u_0\in H^2$, and consequently from~\eqref{esti-ap} it remains true for $u_0\in L^2$.
Now, if for some $ T > 0 $,
\begin{equation} 
\label{eq:u00}
\| u_0 \|_{L^2( \T^2)} \leq C \int_0 ^T \int_{\T^2} a(x)  |u|^2(t,x) dx dt ,
\end{equation} 
where $u$ is the solution of~\eqref{damped},
then \eqref{damping} and semigroup property show that
$  \| u ( k T ) \|_{ L^2 ( \TT^2 )}^2 \leq 
( 1 - 1/C )^N \| u_0 \|_{ L^2 ( \TT^2 ) }^2$, 
and the exponential decay~\eqref{damped} follows.  

For any fixed $T>0$ \eqref{eq:u00} 
is the same as~\eqref{eq:th1} with $ W = a^{\frac12} $, except that here $u$ is the solution of the damped Schr\"odinger equation, while in~\eqref{eq:th1} it is the solution of the free Schr\"odinger equation. 

We now claim that \eqref{eq:th1}, $ W = a^{\frac12} $, implies 
\eqref{eq:u00}. In fact, suppose that \eqref{eq:u00} is not true.
Then, there exists a sequence $ u_{0,n} \in L^2 $, 
\[  \| u_{0,n} \|_{L^2} =1, \  \ \ ( i \partial_t  + \Delta )u_n = a u_n , \ \ u_n|_{t=0} = u_{0,n}, \ \
 \| a^{\frac12} u_n \|_{ L^2 ( ( 0 , T ) \times \TT^2 ) }  \to 0 , \ 
 n\to \infty. \]
Then  
\[ \| a u_n \|_{ L^{\frac43} ( \TT^2; L^2 ( ( 0 , T ) )) } =
\| a^{\frac12} a^{\frac12} u \|_{ L^{\frac43} ( \TT^2; L^2 ( ( 0 , T ) )) }
\leq \| a^{\frac12} \|_{L^4( \TT^2 )  }
\| a^{\frac12} u_n \|_{ L^2 ( ( 0 , T ) \times \TT^2 ) }  \to 0  ,\]
and Theorem \ref{th.BBZ} shows that 
$ u_n = e^{ it \Delta} u_{0,n} + e_n$, 
%\| e_n\|_{ L^\infty ( ( 0, T ) , L^2 ( \TT^2 ) ) } + 
$ \| e_n \|_{ L^4 ( \TT^2 ; L^2 ( ( 0 , T ))) } \to 0 $. 
But then, using \eqref{eq:th1},
\[ \begin{split} 0 & = \limsup_{n \to \infty } \| a^{\frac12} u_n \|_{L^2 ( ( 0, T ) \times \TT^2 ) } 
\\ & 
\geq \limsup_{n \to \infty }  \left( \| a^{\frac12} e^{ it \Delta }
u_{ 0,n} \|_{ L^2 ( ( 0, T ) \times \TT^2 ) } - \| a^{\frac12} \|_{L^4 }
\| e_n \|_{ L^4 ( \TT^2 ; L^2 ( ( 0 , T ))) } \right) \\
& \geq
\limsup_{n \to \infty } \| a^{\frac12} e^{ it \Delta }
u_{ 0,n} \|_{ L^2 ( ( 0, T ) \times \TT^2 ) } \geq
c
\limsup_{n\to \infty } \| u_{0,n}\|_{L^2 ( \TT)} 
%\\ & 
= c > 0 \end{split}\]
which gives a contradition. Hence \eqref{eq:u00} holds and that completes the proof.
\end{proof}

\section*{Appendix}

\refstepcounter{section}
\renewcommand{\thesection}{A}

To see that Theorem \ref{th.1} for $ T > \pi $ and rational tori 
follows from \cite[Theorem 1.2]{Jak}
assume that $ \TT^2 = 
(\RR/2\pi \ZZ )^2 $. We then write $ u ( z ) = \sum_{ \lambda  } u_\lambda $, where
the sum of is over distinct eigenvalues of $ - \Delta $ (and $u_\lambda$ is the projection of $u$ on the corresponding eigenspace).
By Ingham's inequality \cite{ing}
(this is where $ T > \pi $ is used),
\[ \int_0^T \| W e^{ it \Delta} u \|^2_{ L^2 ( \TT^2 )}  =
\int_{\TT^2 } \int_0^T \left| \sum_{ \lambda \in \NN } W(z) u_\lambda (z)   
e^{i t \lambda } \right|^2 dt dz \geq B 
\int_{\TT^2} \sum_{ \lambda \in \NN }  | W ( z ) u_\lambda ( z ) |^2 dz.\]
Hence, \eqref{eq:th1} follows from the estimate,
\begin{equation}
\label{eq:Jak}  \sum_{ \lambda } \| u_\lambda\|^2_{L^2 ( \TT^2 ) } \leq C \int_{\TT^2} \sum_{ \lambda \in \NN }  | W ( z ) u_\lambda ( z ) |^2 dz, \end{equation}
which it turn follows from a pointwise estimate: 
\begin{equation}
\label{eq:Jak1} 
  \| u_\lambda\|^2_{L^2 ( \TT^2 ) } \leq C \int_{\TT^2} | W ( z ) u_\lambda ( z ) |^2 dz, \ \ \ 
 - \Delta u_\lambda = \lambda u_\lambda . \end{equation}

\begin{proof}[Proof of \eqref{eq:Jak1}] 
We start with the observation that the zero set of a non-trivial trigonometric polynomial $ p ( z ) $ has 
measure zero and hence,
\begin{equation}
\label{eq:trig}
\int_{\TT^2} | W ( z ) p( z ) |^2 dz > 0 .
\end{equation} 
In particular that holds for any fixed eigenfunction of $ - \Delta $.

To prove \eqref{eq:Jak1} we proceed by contradiction, that is we assume
that there exists a sequence of $ e_n $'s, such that
\begin{equation} 
\label{eq:en} \|e_n \|_{L^2}^2 = 1 , \ \ \ \| W e_n \|_{L^2}^2 \rightarrow 0
, \ \  - \Delta e_n = 
\lambda_n e_n .  \end{equation}

Suppose first that $ \lambda_n $ are bounded. We can then assume that
$ \lambda_n \to \lambda $. From \eqref{eq:en} we see that 
$ e_n $ are bounded in $ H^2 $ and hence we can assume that 
$ e_n \to e $ in $ H^1 $ and, as $ H^1 \subset L^4 $, also in $ L^4 $. Then \eqref{eq:en} shows that $ -\Delta e = \lambda e$, $ \| e\|_{L^2 } = 1 $, $ \| W e \|_{L^2 } = 0$, which contradicts~\eqref{eq:trig}. 

Hence we can assume (by extracting a subsequence) that $ \lambda_n \to \infty $ in \eqref{eq:en}. We can then assume that 
the sequence of probability measures $ |e_n|^2 dx $ converges weakly to 
a measure $ \nu$. 
According to \cite[Theorem 1.2]{Jak}, $ \nu =  p ( z ) dz $ where
$ p$ is a non-negative trigonometric polynomial, $ \int p ( z ) dz = 1 $.

 Let $ f_k\in C^0$, 
$ f_k \geq 0$,  converge to $|W|^2 $ in $L^2$.  From Zygmund's bound on the $L^4$ norm of $e_n$~\eqref{eq:zyg0},  we get 
$$ \limsup_{n\rightarrow + \infty} | \int (f_k - |W |^2)|e_n|^2 (x) dx | \leq C \| f_k - |W|^2\|_{L^2} ,$$
and from the weak convergence 
$ \lim_{n\rightarrow + \infty}  \int f_k|e_n|^2 (x) dx = \int f_k (x) p(x)dx.$ 
We deduce 
$$0= \lim_{n\rightarrow + \infty}  \int |W e_n|^2 (x) dx = \int |W (x)|^2 p(x)dx. $$ This again contradicts \eqref{eq:trig}.
\end{proof}

\end{document}